\documentclass[a4paper,10pt]{amsart}
\usepackage{amssymb}
\usepackage{color}
\usepackage{latexsym}
\usepackage{amsfonts}
\usepackage{amsmath}
\usepackage{amssymb}
\usepackage{mathtools}
\usepackage{amsthm}
\usepackage{amsxtra}
\usepackage{hyperref}
 \usepackage{amscd}
 \usepackage{mathrsfs}
 \usepackage{pb-diagram}
 \usepackage{comment}
 \usepackage[english]{babel}
 \usepackage{graphicx}
\usepackage{srcltx}
\usepackage{MnSymbol}

\newcommand{\R}{\mathbb{R}}

\newcommand{\Or}{\mathrm{O}}

\newcommand{\D}{\partial}
\newcommand{\X}{\times} 
\renewcommand{\t}[1]{\widetilde{#1}}

\newcommand{\rank}{\mathop\mathrm{rank}\nolimits}
\newcommand{\im}{\mathop\mathrm{im}\nolimits}

\newcommand{\T}{\downvdash}
\newcommand{\F}{\mathcal{F}}

\newcommand{\dif}{\mathrm{\,d}}

\newcommand{\f}{\mathfrak{f}}
\newcommand{\g}{\mathfrak{g}}
\newcommand{\h}{\mathfrak{h}}

\newcommand{\xbf}{\boldsymbol{\mathrm{x}}}

\newtheorem{theorem}{Theorem}[section]
\newtheorem{corollary}[theorem]{Corollary}
\newtheorem{lemma}[theorem]{Lemma}
\newtheorem{example}[theorem]{Example}
\newtheorem{proposition}[theorem]{Proposition}
\newtheorem{remark}[theorem]{Remark}

\numberwithin{equation}{section}

\title[Periodic solutions of semi-explicit DAEs with
time-dependent\ldots ] {Periodic solutions of semi-explicit
  differential-algebraic equations with time-dependent constraints}

\author{Luca Bisconti, Alessandro Calamai and Marco Spadini}
\address[L.\ Bisconti and M.\ Spadini]{Dipartimento di Matematica e
  Informatica, Universit\`a di Firenze, Via S.\ Marta 3, 50139
  Firenze, Italy} \address[A. Calamai]{Dipartimento di Ingegneria
  Industriale e Scienze Matematiche, Universit\`a Politecnica delle
  Marche, Via Brecce Bianche, 60131 Ancona, Italy}

\begin{document}

\begin{abstract} In this paper we investigate the properties of the
  set of $T$-periodic solutions of semi-explicit parametrized
  Differential-Algebraic Equations with non-autonomous constraints of
  a particular type.  We provide simple, degree-theoretic conditions
  for the existence of branches of $T$-periodic solutions of the
  considered equations.  Our approach is based on topological
  arguments about differential equations on implicitly defined
  manifolds, combined with elementary facts of matrix analysis.
\end{abstract}
\maketitle


\section{Introduction}
Several mathematical models arising from physical and engineering
problems can be described in terms of differential-algebraic equations
(DAEs). Because of this, in recent years, there has been a lot of
interest on these equations from both the point of view of the pure
and applied mathematicians. Beside the more genuinely modelistic or
numerical approaches, there are many books and papers that treat DAEs
from an analytical perspective. Of all those, in order to avoid an
impossibly long and necessarily incomplete list, we only mention
\cite{dae, RR0, R:1988} and references therein.

\smallskip A relevant case is represented by first order
\emph{semi-explicit DAEs in Hessenberg form} (see, e.g., \cite{dae})
that is:
\begin{equation} \label{eq:Hessenberg} \left\{ \begin{array}{l}
      \dot x = f (t, x, y), \\
      \mathcal{G} (t, x, y) =0,
    \end{array} \right.
\end{equation}
where $f \colon\R\times \R^m\times\R^s\to\R^m$ is a continuous map,
and $\mathcal{G}\colon\R\times\R^m\times\R^s\to\R^s$ is sufficiently
smooth. If we assume that the partial derivative, $\D_3\mathcal{G}$,
of $\mathcal{G}$ with respect to the third variable $y$ is invertible,
then \eqref{eq:Hessenberg} is said to be of \emph{index 1}.

In this paper we are concerned with a parametrized special case of
\eqref{eq:Hessenberg}.  In fact, we assume that the constraint
$\mathcal{G}$ has the form
\[
\mathcal{G}(t,x,y)=g\big(A(t)x, B(t)y\big)
\]
where $g\colon \R^m\times\R^s\to\R^s$ is $C^\infty$, and the
square-matrix-valued maps $A\colon\R\to\mathrm{O}(\R^m)$ and
$B\colon\R\to\mathrm{GL}(\R^s)$ are continuous.  Here
$\mathrm{O}(\R^m)$ denotes the group of orthogonal $m\X m$ matrices
and $\mathrm{GL}(\R^s)$ the group of $s\X s$ invertible ones.

Namely, for $\lambda\geq 0$ we consider parametrized DAEs of the
following form
\begin{equation} \label{eq:Hessenberg-perturbed}
  \left\{ \begin{array}{l}
      \dot x =  \lambda f(t, x, y), \,\, \lambda\geq 0,\\
      g\big(A(t)x, B(t)y\big) =0,
    \end{array} \right.
\end{equation}
with $f$ as in \eqref{eq:Hessenberg}, and we assume that
$\partial_2g(x,y)$ is invertible for all $(x,y)\in\R^m\times\R^s$ and
(for technical reasons) that $A$ is of class $C^1$.

We also treat, in parallel, the following parametrized second order
DAEs
\begin{equation} \label{eq:second-order-perturbed}
  \left\{ \begin{array}{l}
      \ddot x =  \lambda f(t, x, y, \dot x, \dot y), \,\, \lambda\geq 0,\\
      g\big(A(t)x, B(t)y\big) =0.
    \end{array} \right.
\end{equation}
In this case we assume the matrix-valued maps $A$ and $B$ to be of
class $C^2$ and $C^1$, respectively.  The latter type of equations, in
particular, may be used to represent some nontrivial physical systems
as, for instance, constrained systems (see e.g.\ \cite{RR0}).

\smallskip We will assume throughout the paper that the matrix-valued
function $A$ satisfies the following property:
\begin{equation}\label{propA}
  \text{$A(t)\dot A(t)^\T$ is constant $\forall t\in\R$.}
\end{equation}
This assumption might seem unnatural, but it is not so. To understand
why, consider the case when $m=3$. In that case, if $\{e_1,e_2,e_3\}$
is a fixed reference frame in $\R^3$ and $\mathcal{T}(t)=
\{A(t)e_1,A(t)e_2,A(t)e_3\}$ is a moving frame, our assumption is
equivalent to imposing that the angular velocity of $\mathcal{T}$ is
the zero vector. This is, in fact, an immediate consequence of the
definition of angular velocity. An entirely similar statement holds
for $m=2$.

Furthermore, in this paper we will always assume that, for some given
$T>0$, the map $f$ is $T$-periodic in the first variable and that $A$
and $B$ are $T$-periodic.  Following the approach of \cite{cala, CS1,
  CS2, spaDAE}, we study qualitative properties of the set of
$T$-periodic solutions of \eqref{eq:Hessenberg-perturbed} and
\eqref{eq:second-order-perturbed}. Roughly speaking, we show the
existence of an unbounded connected set of ``nontrivial'' $T$-periodic
solutions of \eqref{eq:Hessenberg-perturbed} or
\eqref{eq:second-order-perturbed} emanating from the set of its
constant solutions.  Precise statements will be given in Subsection
\ref{MainResOrder1} for first order equations and in Subsection
\ref{MainResOrder2} for second order ones.  We also show, through some
examples and remarks, how our constructions can be extended to include
several equations of different forms.

Our continuation results are in the spirit of analogous ones by Furi
and Pera for parametrized first- and second-order equations on
differentiable manifolds (for more details see the survey
\cite{fupespa}) and could be considered, in some sense, as
consequences of recent results obtained by the last two authors in
\cite{cala, CS1, CS2, spaDAE}.  However, we wish to point out the
following facts.  First of all, while the continuation results on
differentiable manifolds by Furi and Pera require the knowledge of the
degree (often called characteristic or rotation) of suitable tangent
vector fields, here (as in \cite{cala, CS1, CS2, spaDAE}) we give
conditions only in terms of the well-known Brouwer degree which is
also easier to compute explicitly.  On the other hand, in the present
paper we tackle the case time-dependent constraints (even if of a
peculiar form). In other words, our results can be regarded as
concerning ODEs on particular \emph{$T$-periodically moving} manifolds
defined implicitly.  As far as we know, the techniques of Furi and
Pera have never been applied to moving manifolds, and this novelty is
our main original contribution to the subject.

% Further, we will show how our results apply to other classes of
% equations. For instance, the analysis developed can be used in order
% to study semi-implicit equations of the following type
% \begin{equation*}
%   E\dot\xbf= F(t)\xbf + \lambda C(t)S(\xbf),\,\, \lambda \geq 0,
% \end{equation*}
% \textcolor{green}{assuming that $E\in \R^{n\X n}$ an that $F\,
% \colon\, \R \to \R^{n\X n}$, $C\, \colon\, \R\to \R^{n\X n}$ and $S
% \, \colon\, \R^n\to \R^{n}$ are continuous maps with a particular
% form that will be specified in the following
% (cf. \eqref{DAE:impeqex} below). }
%  % Further,
% % we also assume that $F$ and $C$ are $T$-periodic, $T>0$ given.

This paper is organized as follows. In Section \ref{sec:preliminary}
we collect the preliminaries needed to approach the DAEs in
\eqref{eq:Hessenberg-perturbed} and
\eqref{eq:second-order-perturbed}. In Section \ref{sec:main-results}
we give our main results and we get topological information on the set
of $T$-periodic pairs to the considered equations; examples of
applications of our methods are provided. Finally, in Section
\ref{sec:technical-lemmas}, we give the proofs of the technical
results of matrix analysis used throughout the paper.

\subsection*{Notation}
Throughout the paper, $C_T(\R^k)$ will denote the Banach space of all
the $T$-periodic continuous maps $\zeta\colon\R\to \R^k$ with the
usual supremum norm, and $C_T^1(\R^k)$ will be the Banach space of all
the $T$-periodic $C^1$ maps $\zeta\colon\R\to \R^k$ with the $C^1$
norm.

\pagebreak

\section{Preliminary results} \label{sec:preliminary}

\subsection{First order DAEs} \label{prel1}

Let us consider semi-explicit DAEs, depending on a parameter $\lambda
\geq 0$, of the following forms:

\begin{equation}\label{dae1}
  \left\{
    \begin{array}{l}
      \dot x = \f(x,y)+\lambda \h(t,x,y),\\
      \g(x,y)=0,
    \end{array}\right.
\end{equation}
and
\begin{equation}\label{dae2}
  \left\{
    \begin{array}{l}
      \dot x = \lambda \h(t,x,y),\\
      \g(x,y)=0,
    \end{array}\right.
\end{equation}
where we assume that $\f\colon\R^m\times\R^s\to\R^m$ and
$\h\colon\R\times\R^m\times\R^s\to\R^m$ are continuous maps, $\h$ is
$T$-periodic in the first variable, and
$\g\colon\R^m\times\R^s\to\R^s$ is $C^\infty$ and such that
$\partial_2\g(x,y)$ is invertible for all $(x,y)$. Notice that,
consequently, $\mathfrak{M}:=\g^{-1}(0)$ is a closed submanifold of
$\R^m\times\R^s$.  Furthermore observe that, even if \eqref{dae2} can
be considered as a particular case of \eqref{dae1} (i.e.\ with
$\f(x,y)=0$ identically), for our purposes the two equations need to
be treated separately.

By a \emph{solution} of \eqref{dae1} we mean a pair of $C^1$ functions
$x$ and $y$ defined on an interval $I$ with the property that the
following equalites hold for all $t\in I$: $\dot x(t) =
\f\big(x(t),y(t)\big) + \lambda \h\big(t,x(t),y(t)\big)$ and
$\g\big(x(t),y(t)\big)=0$.  The notion of solution of \eqref{dae2} is
analogous. Notice that one might wish to ask only the continuity of
$y$. In fact, if $x$ is $C^1$, the assumptions on $\g$ together with
the implicit function theorem imply that $y$ is $C^1$.

In this section we recall two results from \cite{cala} and
\cite{spaDAE} (see also \cite{CS1}) about the sets of \emph{$T$-pairs}
of \eqref{dae1} and of \eqref{dae2}, namely, of those pairs
$\big(\lambda; (x,y)\big)\in [0,\infty )\times C_T(\R^m\times\R^s)$
with $(x,y)$ a $T$-periodic solution of \eqref{dae1} and of
\eqref{dae2}, respectively. Recall that a $T$-pair $\big(\lambda;
(x,y)\big)$ of \eqref{dae1} or of \eqref{dae2} is said to be
\emph{trivial} if $\lambda=0$ and $(x,y)$ is constant.

For the sake of simplicity we make some conventions. We will regard
every space as its image in the following diagram of natural
inclusions
\begin{equation}
  \begin{CD}
    \R^m\times\R^s & @>>> & C_T(\R^m\times\R^s) \\
    @VVV &  & @VVV \\
    \left[ 0,\infty \right) \times \R^m\times\R^s & @>>> & \left[
      0,\infty \right) \times C_T(\R^m\times\R^s)
  \end{CD}
  \label{diag1}
\end{equation}
In particular, we will identify $\R^m\times\R^s$ with its image in
$C_T(\R^m\times\R^s)$ under the embedding which associates to any
$(p,q)\in \R^m\times\R^s$ the map $(\bar p,\bar q) \in
C_T(\R^m\times\R^s)$ constantly equal to $(p,q)$.  Moreover we will
regard $\R^m\times\R^s$ as the slice $\{0\}\times
\R^m\times\R^s\subset [0,\infty)\times \R^m\times\R^s$ and,
analogously, $C_T(\R^m\times\R^s)$ as $\{0\}\times
C_T(\R^m\times\R^s)$.  We point out that the images of the above
inclusions are closed.

For simplicity, given $\Omega\subseteq [0,\infty)\times
C_T(\R^m\times\R^s)$, we will denote by $\Omega_\#$ the set consisting
of all the constant functions $(\bar p,\bar q)$ with $\big(0;(\bar
p,\bar q)\big)\in\Omega$.  We will regard $\Omega_\#$ as a subset of
$\R^m\times\R^s$.

The following is a consequence of Theorem 5.1 in \cite{spaDAE}.

\begin{theorem}\label{tunoA}
  Let $\f$, $\h$, $\g$ be as above.  Define $\F\colon
  \R^m\times\R^s\to\R^m\times\R^s$ by
  \[
  \F(x,y)=\big(\f(x,y),\g(x,y)\big).
  \]
  Let $\Omega\subseteq[0,\infty)\times C_T(\R^m\times\R^s)$ be open
  and assume that $\deg(\F, \Omega_\#)$ is well defined and
  nonzero. Then there exists a connected set $\Gamma$ of nontrivial
  $T$-pairs for \eqref{dae1} that meets
  $\mathcal{F}^{-1}(0)\cap\Omega$ and cannot be both bounded and
  contained in $\Omega$.
\end{theorem}

The following is a consequence of Theorem 2.2 in \cite{cala}.

\begin{theorem}\label{tunoB}
  Let $\h$ and $\g$ be as above.  Define $\omega\colon
  \R^m\times\R^s\to\R^m\times\R^s$ by
  \[
  \omega(x,y)=\left(\frac{1}{T}\int_0^T \h(x,y)\dif t,\g(x,y)\right).
  \]
  Let $\Omega\subseteq [0,\infty)\times C_T(\R^m\times\R^s)$ be open
  and assume that $\deg(\omega,\Omega_\#)$ is well-defined and
  nonzero. Then there exists a connected set $\Gamma$ of nontrivial
  $T$-pairs for \eqref{ord2dae2} that meets $\omega^{-1}(0)\cap\Omega$
  and cannot be both bounded and contained in $\Omega$.
\end{theorem}

\subsection{Second order DAEs}\label{GenIIord}
Consider the following second order parametrized DAEs:

\begin{equation}\label{ord2dae1}
  \left\{
    \begin{array}{l}
      \ddot x = \f(x,y,\dot x,\dot y)+\lambda \h(t,x,y,\dot x,\dot y),\\
      \g(x,y)=0,
    \end{array}\right.
\end{equation}
and
\begin{equation}\label{ord2dae2}
  \left\{
    \begin{array}{l}
      \ddot x = \lambda \h(t,x,y,\dot x,\dot y),\\
      \g(x,y)=0,
    \end{array}\right.
\end{equation}
where we assume that $\f\colon
\R^m\times\R^s\times\R^m\times\R^s\to\R^m$ and $\h\colon\R\times
\R^m\times\R^s\times\R^m\times\R^s\to\R^m$ are continuous maps, $\h$
is $T$-periodic in the first variable, and $\g\colon
\R^m\times\R^s\to\R^s$ is $C^\infty$ and such that $\partial_2\g(x,y)$
is invertible for all $(x,y)$.

By a solution of \eqref{ord2dae1} we mean a pair of $C^2$ functions
$x$ and $y$ defined on an interval $I$ with the property that the
following equalites hold for all $t\in I$: $\ddot x(t)
=\f\big(t,x(t),y(t),\dot x(t),\dot
y(t)\big)+\lambda\h\big(t,x(t),y(t),\dot x(t),\dot y(t)\big)$ and
$\g\big(x(t),y(t)\big)=0$.  Notice that, as in the first order case,
it is equivalent to ask only the continuity of $y$.

\begin{remark}
  We wish to point out that, despite their similarity, it might not be
  possible to reduce second order equations, such as \eqref{ord2dae1}
  or \eqref{ord2dae2}, to first order ones, as \eqref{dae1} or
  \eqref{dae2}.  Such is the case expecially when there is an explicit
  dependence on $\dot y$.  Thus the latter need a specific study.

  Consider for instance equation \eqref{ord2dae2}. The introduction of
  a new variable $u=\dot x$, as it is customary in phase-space
  techniques, would not reduce it to an equation of the form
  \eqref{dae1} or \eqref{dae2} with the required properties. In fact,
  we would get an equation of the type
  \[
  \left\{ \begin{array}{l}
      \dot x = u,\\
      \dot u =  \lambda f(t, x, y, u, \dot y), \\
      0= \bar g\big((x,u),y\big),
    \end{array} \right.
  \]
  where $\bar g\big((x,u);y\big)=g(x,y)$, which is not of the form
  \eqref{dae1}, because of the $\dot y$ in the second equation. The
  introduction of another auxiliary variable $v=\dot y$, as it could
  seem natural, would only complicate matters. Indeed, the resulting
  equation would be the following:
  \[
  \left\{ \begin{array}{l}
      \dot x =u,\\
      \dot y =v,\\
      \dot u =  \lambda f(t, x, y, u, v), \\
      0= \hat g\big((x,u);(y,v)\big),
    \end{array} \right.
  \]
  where $\hat g\big((x,u);(y,v)\big)=g(x,y)$. What is wrong with this
  equation is that the rigid dimensional separation between the
  ``differential'' and the ``algebraic'' parts required for
  \eqref{dae1} is now broken.
\end{remark}

The structure of the set of solution pairs of \eqref{ord2dae1} and of
\eqref{ord2dae2} has been studied in \cite{CS2}. As in Section
\ref{prel1}, we recall that by a \emph{$T$-pair} of \eqref{ord2dae1}
and of \eqref{ord2dae2} we mean a pair $\big(\lambda; (x,y)\big)\in
[0,\infty )\times C_T^1(\R^m\times\R^s)$ with $(x,y)$ a $T$-periodic
solution of \eqref{ord2dae1} and of \eqref{ord2dae2}, respectively.
Again, a $T$-pair $\big(\lambda; (x,y)\big)$ of \eqref{ord2dae1} or of
\eqref{ord2dae2} is said to be \emph{trivial} if $\lambda=0$ and
$(x,y)$ is constant.

As in Section \ref{prel1}, for simplicity we will regard every space
as its image in the following diagram of natural inclusions
\begin{equation}
  \begin{CD}
    \R^m\times\R^s & @>>> & C_T^1(\R^m\times\R^s) \\
    @VVV &  & @VVV \\
    \left[ 0,\infty \right) \times \R^m\times\R^s & @>>> & \left[
      0,\infty \right) \times C_T^1(\R^m\times\R^s)
  \end{CD}
  \label{diag2}
\end{equation}
with the obvious analogous identifications.

Again, given $\Omega\subseteq [0,\infty)\times C_T^1(\R^m\times\R^s)$,
we will denote by $\Omega_\#$ the set consisting of all the constant
functions $(\bar p,\bar q)$ with $\big(0;(\bar p,\bar
q)\big)\in\Omega$ and will regard $\Omega_\#$ as a subset of
$\R^m\times\R^s$.

The next results are straightforward consequences of Corollary 5.2 and
Corollary 5.3 in \cite{CS2}, respectively.

\begin{theorem}\label{tdueA}
  Let $\f$, $\h$, $\g$ be as above.  Define $\F\colon
  \R^m\times\R^s\to\R^m\times\R^s$ by
  \[
  \F(x,y)=\big(\f_0(x,y),\g(x,y)\big),
  \]
  where $\f_0(x,y):=\f(x,y,0,0)$.  Let
  $\Omega\subseteq[0,\infty)\times C_T^1(\R^m\times\R^s)$ be open and
  assume that $\deg(\F, \Omega_\#)$ is well defined and nonzero. Then
  there exists a connected set $\Gamma$ of nontrivial $T$-pairs for
  \eqref{ord2dae1} that meets $\mathcal{F}^{-1}(0)\cap\Omega$ and
  cannot be both bounded and contained in $\Omega$.
\end{theorem}

\begin{theorem}\label{tdueB}
  Let $\h$ and $\g$ be as above.  Define $\omega\colon
  \R^m\times\R^s\to\R^m\times\R^s$ by
  \[
  \omega(x,y)=\left(\frac{1}{T}\int_0^T \h_0(x,y)\dif
    t,\g(x,y)\right),
  \]
  where $\h_0(x,y):=\h(x,y,0,0)$.  Let $\Omega\subseteq
  [0,\infty)\times C_T(\R^m\times\R^s)$ be open and assume that
  $\deg(\omega,\Omega_\#)$ is well-defined and nonzero. Then there
  exists a connected set $\Gamma$ of nontrivial $T$-pairs for
  \eqref{ord2dae2} that meets $\omega^{-1}(0)\cap\Omega$ and cannot be
  both bounded and contained in $\Omega$.
\end{theorem}

\section{Coordinate transformation and main
  results} \label{sec:main-results}

\subsection{First order DAEs}\label{MainResOrder1}
 
We first investigate parametrized DAEs of the following form:
\begin{equation}\label{DAE_ord1_vm}
  \left\{
    \begin{array}{l}
      \dot x = \lambda f(t,x,y),\; \lambda\geq 0,\\
      g\big(A(t)x,B(t)y\big)=0
    \end{array}
  \right.
\end{equation}
where, as in the introduction, the map $f\colon \R\X\R^m\X\R^s\to
\R^m$ is continuous and $T$-periodic in the first variable, $g\colon
\R^m\times\R^s\to\R^s$ is $C^\infty$ and such that $\partial_2
g(\xi,\eta)$ is invertible for all $(\xi,\eta)$, and
$A\colon\R\to\mathrm{O}(\R^m)$ and $B\colon\R\to\mathrm{GL}(\R^s)$ are
$T$-periodic continuous (square-)matrix-valued maps.  We will assume
that $A$ is of class $C^1$.

Let us apply, for all $t$, a change of coordinates in $\R^m\X\R^s$:
\begin{equation}\label{coord-transf}
  \xi(t)=A(t)x(t),\quad \eta(t)=B(t)y(t).
\end{equation}
Let us rewrite the first of these two equations as
$x(t)=A^\T(t)\xi(t)$. Differentiating with respect to $t$ we get
\begin{equation}\label{xinxi}
  \dot x(t) = \dot A(t)^\T \xi(t) + A(t)^\T \dot \xi(t).
\end{equation}
Observe, in fact, that the operations of differentiation and
transposition commute; that is:
\[
\left(\dot A(t)\right)^\T=\frac{d}{dt}\left(A(t)^\T\right).
\] {}From \eqref{xinxi} we get $\dot \xi(t) = - A(t)\dot A(t)^\T\xi(t)
+ A(t)\dot x(t)$.  Thus, equation \eqref{DAE_ord1_vm} can be rewritten
in the new coordinates $(\xi,\eta)$ as follows:
\begin{equation}\label{DAE_ord1_vf}
  \left\{
    \begin{array}{l}
      \dot \xi = - A(t)\dot A(t)^\T\xi +\lambda F(t,\xi,\eta),\; 
      \lambda\geq 0,\\
      g(\xi,\eta)=0
    \end{array}
  \right.
\end{equation}
where $F\colon \R\X\R^m\X\R^s\to \R^m$ is defined by
\begin{equation}\label{def-F}
  F(t,\xi,\eta)= A(t)f\big(t,A(t)^\T \xi,B^{-1}(t) \eta \big).
\end{equation}
If we assume that the matrix $M:=A(t)\dot A(t)^\T$ is constant, then
we can obtain continuation results for $T$-pairs of
\eqref{DAE_ord1_vm} as consequences of the results in the previous
section.

In the following we will adopt the same notation as in Section
\ref{prel1}.

\begin{theorem}\label{thm.DAE.1ord}
  Let $f,g,A$ and $B$ be as above. Assume that $M:=A(t)\dot A(t)^\T$
  is constant and define
  $\mathcal{F}\colon\R^m\times\R^s\to\R^m\times\R^s$ by
  $\mathcal{F}(x,y)=\big(Mx,g(x,y)\big)$. Let $\Omega\subseteq
  [0,\infty)\times C_T(\R^m\times\R^s)$ be open and assume that
  $\deg(\mathcal{F},\Omega_\#)$ is well-defined and nonzero. Then
  there exists a connected set $\Gamma$ of nontrivial $T$-pairs for
  \eqref{DAE_ord1_vm} that meets $\mathcal{F}^{-1}(0)\cap\Omega$ and
  cannot be both bounded and contained in $\Omega$.
\end{theorem}
\begin{proof}
  Consider the transformation \eqref{coord-transf}. As discussed
  above, in the new coordinates $\xi,\eta$ equation
  \eqref{DAE_ord1_vm} becomes \eqref{DAE_ord1_vf}, which we write as
  \begin{equation}\label{DAEo1tmp}
    \left\{\begin{array}{l}
        \dot\xi=-M\xi+\lambda F(t,\xi,\eta),\\
        g(\xi,\eta)=0,
      \end{array}\right.
  \end{equation}
  where $F$ is defined as in \eqref{def-F}. Consider also the
  homeomorphism $\mathfrak{H}\colon[0,\infty)\times
  C_T(\R^m\times\R^s)\to[0,\infty)\times C_T(\R^m\times\R^s)$ given by
  $\mathfrak{H}\big(\lambda,(x,y)\big)=\big(\lambda,(\xi,\eta)\big)$
  with $\xi$ and $\eta$ given by \eqref{coord-transf}. Clearly
  $\mathfrak{H}$ establish a homeomorphism between the space $X$ of
  $T$-pairs of \eqref{DAEo1tmp} and the space $\mathcal{X}$ of
  $T$-pairs of \eqref{DAE_ord1_vf}, which preserves triviality. In the
  sense that $\mathfrak{H}$ takes trivial $T$-pairs of
  \eqref{DAEo1tmp} to trivial ones of \eqref{DAE_ord1_vf} and, vice
  versa, $\mathfrak{H}^{-1}$ makes trivial $T$-pairs of
  \eqref{DAE_ord1_vf} correspond to trivial ones of \eqref{DAEo1tmp}.

  Let $\mathcal{W}=\mathfrak{H}(\Omega)$. Applying Theorem \ref{tunoA}
  we get the existence of a connected set, let us say $\Upsilon$, of
  nontrivial $T$-pairs for \eqref{DAEo1tmp} that meets
  $\mathcal{F}^{-1}(0)\cap\mathcal{W}$ and cannot be both bounded and
  contained in $\mathcal{W}$. One sees immediately that $\Gamma=
  \mathfrak{H}^{-1}(\Upsilon)$ has the required properties.
\end{proof}

In the following consequence of Theorem \ref{thm.DAE.1ord} we further
assume that $M$ is nonsingular and use the properties of Brouwer
degree to get a continuation result with the sole assumption that
$[g(0,\cdot)]^{-1}(0)\cap\Omega_\#$ is a nonempty and compact subset
of $\R^m\times\R^s$.

\begin{corollary}\label{cor.DAE.1ord}
  Let $f,g,A$ and $B$ be as above. Assume that $M:=A(t)\dot A(t)^\T$
  is constant and nonsingular.  Let $\Omega\subseteq [0,\infty)\times
  C_T(\R^m\times\R^s)$ be open. Assume that the set
  $[g(0,\cdot)]^{-1}(0)\cap\Omega_\#$ is nonempty and compact. Then
  there exists a connected set $\Gamma$ of nontrivial $T$-pairs for
  \eqref{DAE_ord1_vm} that meets $[g(0,\cdot)]^{-1}(0)\cap\Omega$ and
  cannot be both bounded and contained in $\Omega$.
\end{corollary}

\begin{proof}
  Let $\mathcal{F}$ be as in the assertion of Theorem
  \ref{thm.DAE.1ord}. Since the first component of $\mathcal{F}$ is
  nonsingular, the reduction property of Brouwer degree implies
  \[
  \deg(\mathcal{F},\Omega_\#) =\det M \cdot
  \deg\Big(g(0,\cdot),\Omega_\#\cap(\{0\}\times\R^s) \Big).
  \]
  Observe now that since $\partial_2g(\xi,\eta)$ is never singular,
  \[
  \left|\deg\Big(g(0,\cdot),\Omega_\#\cap(\{0\}\times\R^s)
    \Big)\right| =\#\Big([g(0,\cdot)]^{-1}(0)\cap\Omega_\#\Big),
  \]
  which is finite and nonzero.
\end{proof}

In the next result we assume $M=0$ and apply Theorem \ref{tunoB}.

\begin{theorem}\label{thm.DAE.1ord2}
  Let $f,g,A$ and $B$ be as above. Assume that $A(t)\dot A(t)^\T$ is
  identically zero and define
  $\omega\colon\R^m\times\R^s\to\R^m\times\R^s$ by
  \[
  \omega(\xi,\eta)=\left(\frac{1}{T}\int_0^TA(t)f\big(t,A(t)^\T\xi,B^{-1}(t)\eta\big)dt,
    g(\xi,\eta)\right).
  \]
  Let $\Omega\subseteq [0,\infty)\times C_T(\R^m\times\R^s)$ be open
  and assume that $\deg(\omega,\Omega_\#)$ is well-defined and
  nonzero. Then there exists a connected set $\Gamma$ of nontrivial
  $T$-pairs for \eqref{DAE_ord1_vm} that meets
  $\omega^{-1}(0)\cap\Omega$ and cannot be both bounded and contained
  in $\Omega$.
\end{theorem}

\begin{proof}
  Follows from Theorem \ref{tunoB} whith the same proof of Theorem
  \ref{thm.DAE.1ord}.
\end{proof}

% \begin{remark}
%   The assumption that $A(t)\dot A(t)^\T$ is constant might seem
%   outlandish at first sight. Actually, it is not so: when $k=3$, if
%   $\{e_1,e_2,e_3\}$ is a fixed reference frame in $\R^3$ and
%   $\mathcal{T}(t)= \{A(t)e_1,A(t)e_2,A(t)e_3\}$ is a moving frame,
%   our assumption is equivalent to imposing that the angular velocity
%   of $\mathcal{T}$ is the zero vector. This is, in fact, an
%   immediate consequence of the definition of angular velocity. An
%   entirely similar statement holds for $k=2$.
% \end{remark}
%
%\begin{example}
%  The matrix function
%  \[
%  A(t):=
%  \begin{pmatrix}
%    \cos t & \sin t\\
%    -\sin t & \cos t
%  \end{pmatrix}
%  \]
%  has the property requested in Theorem \ref{thm.DAE.1ord}. In fact
%  \[
%  A(t)\dot A(t)^\T=
%  \begin{pmatrix}
%    0 & -1\\
%    1 & 0
%  \end{pmatrix}
%  \]
%\end{example}

\begin{example}\label{protante}
  Take $m=2$ and $s=1$. Let $f\colon\R\X\R^2\X\R\to\R^2$ be any
  continuous mapping $2\pi$-periodic in the first variable.  Consider
  \begin{equation}\label{eqprotante}
    \left\{
      \begin{array}{l}
        \dot x = \lambda f(t,x,y),\; \lambda\geq 0,\\
        y^3+y-x_1^2-x_2^2-(x_1\sin t+x_2\cos t)^2=0,
      \end{array}
    \right.
  \end{equation}
  where $x=(x_1,x_2)$. It is readily verified that
  \[
  y^3+y-x_1^2-x_2^2-(x_1\sin t+x_2\cos t)^2=g\big(A(t)x,y\big),
  \]
  where
  \[
  A(t):=
  \begin{pmatrix}
    \cos t & -\sin t\\
    \sin t & \cos t
  \end{pmatrix},\;\text{ and }\; g(p_1,p_2,q)=q^3+q-p_1^2-2p_2^2.
  \]
  Thus, the constraint can be regarded as the surface having equation
  $q^3+q=p_1^2+2p_2^2$, in the space $(p_1,p_2,q)$, revolving around
  the $q$ axis (a full rotation takes time $2\pi$). With the
  transformation \eqref{coord-transf} the above DAE becomes
  \[
  \left\{
    \begin{array}{l}
      \dot \xi = M \xi + \lambda F(t,\xi,\eta),\\
      g(\xi,\eta)=0,
    \end{array}
  \right.
  \]
  where
  \[
  M:=
  \begin{pmatrix}
    0 & 1 \\
    -1 & 0
  \end{pmatrix},\;\text{ and }\; F(t,\xi,\eta)=
  A(t)f\big(t,A(t)^\T\xi,B^{-1}(t)\eta).
  \]
  Let $\Omega=[0,\infty)\X C_T(\R^2\X\R)$. Since the degree in $\R$ of
  the map
  $(\xi_1,\xi_2,\eta)\mapsto\big(\xi_2,-\xi_1,\eta^3+\eta-\xi_1^2-\xi_2^2\big)$
  is equal to $1$, Theorem \ref{thm.DAE.1ord} yields an unbounded
  connected set of nontrivial $2\pi$-pairs for \eqref{eqprotante} that
  meets $\big(0;(0,0;0)\big)\in [0,\infty)\X\R^2\X\R$ (regarded as a
  $2\pi$-pair).
\end{example}

\begin{remark}\label{rem_conH}
  Notice that a similar coordinate transformation applies also to a
  slightly different situation. Consider the following DAE:
  \begin{equation}\label{Iord_conH}
    \left\{\begin{array}{l}
        \dot x = Hx+\lambda f(t,x,y),\; \lambda\geq 0,\\
        g\big(A(t)x,B(t)y\big)=0
      \end{array}\right.
  \end{equation} 
  where $A$, $B$, $f$ and $g$ are as in \eqref{DAE_ord1_vm} and $H$ is
  a matrix that commutes with $A$. Suppose, as above, that $M:=
  A(t)\dot A(t)^\T$ is constant (not necessarily invertible) and apply
  the transformation as indicated above. Equation \eqref{Iord_conH}
  becomes
  \begin{equation}\label{Iord_conH_trasf}
    \left\{
      \begin{array}{l}
        \dot \xi = \big(H- M\big)\xi +
        \lambda F(t,\xi,\eta),\; \lambda\geq 0,\\
        g(\xi,\eta)=0
      \end{array}
    \right.
  \end{equation}
  with $F$ as in \eqref{DAE_ord1_vf}, so that the results of the
  previous section are applicable to \eqref{Iord_conH_trasf}.
\end{remark}

\begin{example}
  Consider the following DAE:
  \begin{equation}\label{eqexconHIord}
    \left\{\begin{array}{l}
        \dot x_1=x_1+\lambda f_1(t,x_1,x_2,y)\\
        \dot x_2=\lambda f_2(t,x_1,x_2,y)\\
        y^5+y=x_1\cos t +x_2\sin t
      \end{array}\right.
  \end{equation}
  where $f_i\colon\R\X\R\X\R\X\R\to\R^2$, $i=1,2$, are continuous
  mappings $2\pi$-periodic in the first variable. If we put
  $x=(x_1,x_2)$ Equation \eqref{eqexconHIord} is of the form
  \eqref{Iord_conH} with
  \[
  H=\begin{pmatrix}1 & 0 \\ 0& 0\end{pmatrix},\quad A(t)
  = \begin{pmatrix}
    \cos t &  \sin t\\
    -\sin t & \cos t
  \end{pmatrix}, \quad B(t) \equiv \begin{pmatrix}
    1 & 0\\
    0 & 1
  \end{pmatrix},
  \]
  and $f\colon\R\X\R^2\to\R^2$ and $g\colon\R^2\X\R\to\R$ defined by
  \[
  f(t,x)=\big(f_1(t,x_1,x_2,y),f_1(t,x_1,x_2,y)\big)\quad\text{ and
  }\quad g(x,y)=y+y^3-x_1,
  \]
  respectively. Clearly, as in Remark \ref{rem_conH}, Equation
  \eqref{eqexconHIord} becomes
  \begin{equation}\label{eqModifC32}
    \left\{
      \begin{array}{l}
        \dot \xi =(H-M) \xi + \lambda F(t,\xi,\eta),\\
        \eta+\eta^5-\xi_1=0,
      \end{array}
    \right.
  \end{equation}
  where
  \[
  M:= \begin{pmatrix}
    0 & -1 \\
    1 & 0
  \end{pmatrix},\;\text{ and }\; F(t,\xi,\eta)=
  A(t)f\big(t,A(t)^\T\xi,B^{-1}(t)\eta\big).
  \]
  Equation \eqref{eqModifC32} is of the form considered in Corollary
  \ref{cor.DAE.1ord}.
\end{example}

\medskip In our next example we consider periodic perturbations of a
class of semi-linear DAEs (semi-linear DAEs find practical
applications in robotics and electrical circuit modeling see e.g.\
\cite{mgerdin, dae}). We will restrict ourselves to the case when the
equation has a particular `\emph{separated variables}' form, that is
\begin{equation}\label{DAE:impeqex}
  E\dot\xbf= F(t)\xbf + \lambda C(t)S(\xbf),
\end{equation}
$F\, \colon\, \R \to \R^{n\X n}$, $C\, \colon\, \R\to \R^{n\X n}$ and
$S \, \colon\, \R^n\to \R^{n}$ are continuous maps. Further, we assume
that $F$ and $C$ are $T$-periodic, $T>0$ given.

\begin{example}\label{exm.1stOrd}
  Consider Equation \eqref{DAE:impeqex} with $n=4$ and
  \begin{subequations}
    \label{DAEs:commutingmatrices}
    \begin{gather*}
      E= \begin{pmatrix}
        1 & 0 & 0 & 0 \\
        0 & 0 & 0 & 0 \\
        0 & 0 & 0 & 1 \\
        0 & 0 & 0 & 0 \\
      \end{pmatrix}, \qquad F(t) = \begin{pmatrix}
        0       & 0  & 0  & 0         \\
        \cos t & 1  & 0  & -\sin t  \\
        0       & 0  & 0  & 0         \\
        \sin t & 0  & 1  & \cos t   \\
      \end{pmatrix}\\
      \intertext{ and } C(t)= \begin{pmatrix}
        2 + \cos t & 1 & 0 & 1 \\
        0 & 0 & 0 & 0 \\
        1 & 3 + \sin t & 2 & 0 \\
        0 & 0 & 0 & 0 \\
      \end{pmatrix},\qquad S(\xbf)=\xbf.
    \end{gather*}
  \end{subequations}
  The following orthogonal matrices
  \begin{equation*}
    P =\begin{pmatrix} 
      1 & 0 & 0 & 0 \\
      0 & 0 & 1 & 0 \\
      0 & 1 & 0 & 0 \\
      0 & 0 & 0 & 1
    \end{pmatrix}\quad \textrm{and}\quad Q =\begin{pmatrix}
      1 & 0 & 0 & 0 \\
      0 & 0 & 1 & 0 \\
      0 & 0 & 0 & 1 \\
      0 & 1 & 0 & 0
    \end{pmatrix}
  \end{equation*}
  realize a singular value decomposition for $E$. In particular, we
  have that
  \begin{gather*}
    P^\T EQ = \left(\begin{array}{cc|cc}
        1 & 0 & 0 & 0\\
        0 & 1 & 0 & 0\\
        \hline
        0 & 0 & 0 &  0 \\
        0 &  0 & 0  & 0 \\
      \end{array}\right) =: 
    \left( \begin{array}{cc} \t E_{1} & 0 \\ 0 & 0
      \end{array} \right), 
    \\
    P^\T F(t)Q =\left(\begin{array}{cc|cc}
        0  & 0 & 0 & 0 \\
        0  & 0 & 0 & 0\\
        \hline
        \cos t & -\sin t & 1 & 0 \\
        \sin t & \cos t  & 0  & 1 \\
      \end{array} \right) =: 
    \left( \begin{array}{cc} 0 & 0 \\ \t F_3(t)  & \t F_4(t) 
      \end{array} \right)
    \\
    P^\T C(t)Q =\left(\begin{array}{cc|cc} 
        2 + \cos t & 1 & 1 & 0 \\
        1          & 0 &   3 + \sin t & 2 \\
        \hline
        0          & 0 & 0 & 0 \\
        0          & 0 & 0 & 0 \\
      \end{array} \right) =: 
    \left(\begin{array}{cc} \t C_{1} (t) & \t C_{2}(t) \\ 
        0  & 0 \end{array}\right).
  \end{gather*}
  Then, setting $\xbf =Q\left(\begin{smallmatrix}x\\
      y\end{smallmatrix}\right)$ with $x,y\in\R^2$ and multiplying
  \eqref{DAE:impeqex} by $P^\T$ on the left, we can rewrite Equation
  \eqref{DAE:impeqex} as
  \begin{equation*}
    P^\T EQ\left(\begin{smallmatrix}\dot x \\ \dot y\end{smallmatrix}\right) = P^\T F(t)
    Q\left(\begin{smallmatrix}x \\ y\end{smallmatrix}\right) 
    +\lambda(P^\T C(t)Q)Q^\T S\big(Q\left(\begin{smallmatrix}x\\ y\end{smallmatrix}\right)\big).
  \end{equation*}
  that is,
  \begin{equation*}
    \begin{pmatrix} \t E_1 & 0 \\ 0 & 0
    \end{pmatrix}
    \begin{pmatrix} \dot x \\
      \dot y \end{pmatrix} =
    \begin{pmatrix} 0 & 0 \\
      \t F_{3}(t) & \t F_{4} (t) \end{pmatrix}
    \begin{pmatrix} x \\
      y\end{pmatrix}
    + \lambda  \begin{pmatrix} \t C_{1} (t) & \t C_{2}(t) \\
      0 & 0 \end{pmatrix}\begin{pmatrix} \t S_1(x,y)\\ \t
      S_2(x, y)
    \end{pmatrix}
  \end{equation*}
  where we have set $Q^\T S(Q \mathbf{x})= \left(\begin{smallmatrix}
      \t S_1(x, y)\\ \t S_2(x, y)\end{smallmatrix}\right)$.  This
  equation can be rewritten as follows:
  \begin{equation*}
    \left\{ \begin{array}{l} 
        \dot x =\lambda {\t E_1}^{-1}\Big(  \t C_1(t) \t S_1(x,y)+ \t C_2(t) \t S_2(x,y)\Big),\\
        y +\t{F}_3(t)x=0 ,
      \end{array} \right. 
  \end{equation*}
  or, in our case, as
  \begin{equation*}
    \left\{ \begin{array}{l} 
        \begin{pmatrix}\dot x_1\\ \dot x_2\end{pmatrix}=
        \lambda \begin{pmatrix}(2+\cos t) x_1+x_2+y_1\\ 
          x_1+(3+\sin t)y_1+2y_2\end{pmatrix} ,\\[4mm]
        \begin{pmatrix}y_1\\ y_2\end{pmatrix} +
        \begin{pmatrix} \cos t & -\sin t\\
          \sin t & \cos t \end{pmatrix}
        \begin{pmatrix}x_1\\ x_2\end{pmatrix}=0 ,
      \end{array} \right. 
  \end{equation*}
  where we have put $x=\left(\begin{smallmatrix}\dot x_1\\ \dot
      x_2\end{smallmatrix}\right)$ and
  $y=\left(\begin{smallmatrix}y_1\\ y_2\end{smallmatrix}\right)$. The
  above DAE is of the form \eqref{DAE_ord1_vm} considered in Theorem
  \ref{thm.DAE.1ord2}. Observe that the map $\omega$ considered there
  is given by
  \[
  \omega(x_1,x_2;y_1,y_2)=\Big(y_1, 3y_1+2y_2,x_1+y_1,x_2+y_2\Big).
  \]
\end{example}

The example considered above is a particular case of a more general
procedure that we now roughly sketch. Take $E$, $F$ and $C$ as in
Equation \eqref{DAE:impeqex}, and let $\rank E =r$. Assume that
$n=2r$, and that
\begin{subequations}\label{tworelations}
  \begin{align}
    &\ker\, C^\T(t) = \ker\, E^\T\quad \forall \,t\in\R,
    \label{first-relations}\\
    &\im\, F(t) = \ker\, E^\T\quad \forall \,t\in\R.
    \label{second-relations}
  \end{align}
\end{subequations}
Let $P$, $Q$ be orthogonal matrices realizing a singular value
decomposition for $E$.  Multiply \eqref{DAE:impeqex} by $P^\T$ on the
left, and put $\xbf =Q\left(\begin{smallmatrix}x\\
    y\end{smallmatrix}\right)$ with $x,y\in\R^r$.  We get, as in
Example \ref{exm.1stOrd},
\begin{equation} \label{svdeq-1} P^\T EQ\left(\begin{smallmatrix}\dot
      x \\ \dot y\end{smallmatrix}\right) = P^\T F(t)
  Q\left(\begin{smallmatrix}x \\ y\end{smallmatrix}\right)
  +\lambda(P^\T C(t)Q)Q^\T S\big(Q\left(\begin{smallmatrix}x\\
      y\end{smallmatrix}\right)\big).
\end{equation}
Since $P$ and $Q$ realize a singular value decomposition of $E$, and
since $E$, $F$ and $C$ satisfy equations \eqref{tworelations}, an
inspection of the proof of \cite[Lemma~5.5]{BiSpa2011} (see also
\cite{BiSpa2012}) shows us that for all $t$,
\begin{equation*}
  P^\T E Q = 
  \begin{pmatrix} \t E_{1} & 0
    \\ 0 & 0
  \end{pmatrix},\,
  P^\T F(t) Q =     \begin{pmatrix} 0 & 0 \\
    \t F_{3}(t) & \t F_{4} (t)
  \end{pmatrix} \textrm{ and } P^\T C(t) Q =
  \begin{pmatrix} \t C_{1} (t) & \t C_{2}(t) \\
    0 & 0
  \end{pmatrix}.
\end{equation*}
Set $\xbf=Q\left(\begin{smallmatrix}x\\ y\end{smallmatrix}\right)$ and
$Q^\T S(Q\xbf)= \left(\begin{smallmatrix} \t S_1(x, y)\\ \t S_2(x,
    y)\end{smallmatrix}\right)$. Then, we can rewrite Equation
\eqref{svdeq-1} as
\begin{equation*}
  \begin{aligned}
    \begin{pmatrix} \t E_1 & 0 \\ 0 & 0
    \end{pmatrix}
    \begin{pmatrix} \dot x \\
      \dot y \end{pmatrix} =
    \left( \begin{array}{cc} 0 & 0 \\
        \t F_{3}(t) & \t F_{4} (t) \end{array} \right)
    \begin{pmatrix} x \\
      y\end{pmatrix}
    + \lambda  \left(\begin{array}{cc} \t C_{1} (t) & \t C_{2}(t) \\
        0 & 0 \end{array}\right)\begin{pmatrix} \t S_1(x,y)\\ \t
      S_2(x, y)
    \end{pmatrix}
  \end{aligned}
\end{equation*}
or, equivalently
\begin{equation*} \left\{ \begin{array}{l} \dot x = \lambda \t
      E_1^{-1} \left( \t C_1 (t) \t S_1(x,y)
        + \t C_2 (t) \t S_2(x, y) \right),\\
      \t{F}_3(t)x + \t{F}_4(t) y=0,
    \end{array}\right.
\end{equation*}
and, if $\t F_3(t)$ is invertible for all $t$,
\begin{equation} \label{forma-comoda2} \left\{ \begin{array}{l} \dot x
      = \lambda \t E_1^{-1} \left( \t C_1 (t) \t S_1(x,y)
        + \t C_2 (t) \t S_2(x, y) \right),\\
      x + [\t{F}_3(t)]^{-1}\t{F}_4(t) y=0,
    \end{array}\right.
\end{equation}
which is of type \eqref{DAE_ord1_vm} with $m=s=r$ if also $\t F_4(t)$
is invertible for all $t$.  \smallskip

\subsection{Second order DAEs}\label{MainResOrder2}
Let us now focus on parametrized second order DAEs and proceed as in
the first order case. Consider
\begin{equation}\label{DAE_ord2_vm}
  \left\{
    \begin{array}{l}
      \ddot x = \lambda f(t,x,y,\dot x,\dot y),\; \lambda\geq 0,\\
      g\big(A(t)x,B(t)y\big)=0
    \end{array}
  \right.
\end{equation}
where $f\colon \R\X\R^m\X\R^s\X\R^m\X\R^s\to \R^m$ is continuous and
$T$-periodic in the first variable, $g\colon \R^m\times\R^s\to\R^s$ is
$C^\infty$ and such that $\partial_2 g(\xi,\eta)$ is invertible for
all $(\xi,\eta)$, and the $T$-periodic matrix-valued maps
$A\colon\R\to\mathrm{O}(\R^m)$ and $B\colon\R\to\mathrm{GL}(\R^s)$ are
of class $C^2$ and $C^1$, respectively.  As in the frst order case we
consider the following change of coordinates for all $t$:
\[
\xi(t)=A(t)x(t),\quad \eta(t)=B(t)y(t).
\]
We can rewrite the first of these equations as $x(t)=A^\T(t)\xi(t)$
and, taking the derivative, we get
\[
\dot x = \dot A(t)^\T \xi + A(t)^\T \dot \xi,\quad \ddot x = \ddot
A(t)^\T\xi + 2\dot A(t)^\T\dot\xi+ A(t)^\T\ddot\xi.
\]
Let us multiply by $A$ on the left the second of these
equations. Reordering (and omitting the explicit dependence on $t$) we
get
\[
\ddot\xi = -A\ddot A^\T\xi -2A\dot A^\T\dot\xi + A\ddot x.
\]
Moreover, since $y(t)=B^{-1}(t)\eta(t)$,
\[
\dot y(t)=\frac{\dif }{\dif
  t}\left[B(t)^{-1}\right]\eta(t)+B^{-1}(t)\dot\eta(t).
\]
Thus we can rewrite our DAE, in the new coordinates, as follows:
\begin{equation}\label{DAE_ord2_vf}
  \left\{
    \begin{array}{l}
      \ddot \xi =-A(t)\ddot A(t)^\T\xi -2A(t)\dot A(t)^\T\dot\xi 
      +\lambda F\big(t,\xi,\eta,\dot\xi,\dot\eta\big),\; \lambda\geq 0,\\
      g(\xi,\eta)=0.
    \end{array}
  \right.
\end{equation}
where $F\colon \R\X\R^m\X\R^s\X\R^m\X\R^s\to \R^m$, defined by
\begin{multline*}
  F(t,\xi, \eta,u,v)=\\
  A(t)f\left(t,A(t)^\T \xi,B^{-1}(t) \eta,\dot A(t)^\T \xi + A(t)^\T
    u, \frac{\dif }{\dif t}\left[B(t)^{-1}\right] \eta+B^{-1}(t)
    v\right)
\end{multline*}
is clearly continuous and $T$-periodic.

Now, by Proposition \ref{prop1} (see Appendix), we have that if
$M:=A(t)\dot A(t)^\T$ is constant (and nonsingular), then $A(t)\ddot
A(t)^\T$ is constant (and nonsingular) as well, as it is equal to
$M^2$.  Thus, as for first-order equations, provided that $A(t)\dot
A(t)^\T$ is constant, this DAE can be treated with the methods of the
previous section.

It is also worth noticing that $\frac{\dif }{\dif
  t}\left[B(t)^{-1}\right]$, which appears in the expression of $F$,
can also be conveniently expressed as $-B(t)^{-1}\dot
B(t)B(t)^{-1}$. This trivial fact is readily established by
differentiating the relation $B(t) B(t)^{-1}=I$.

Proceeding as in the previous subsection, and using Theorems
\ref{tdueA} and \ref{tdueB} in place of Theorems \ref{tunoA} and
\ref{tunoB}, we get the following results, remarkably similar to
Theorems \ref{thm.DAE.1ord} and \ref{thm.DAE.1ord2}, and Corollary
\ref{cor.DAE.1ord}:

\begin{theorem}\label{thm.DAE.2ord}
  Let $f,g,A$ and $B$ be as above. Assume that $M:=A(t)\dot A(t)^\T$
  is constant and define
  $\mathcal{F}\colon\R^m\times\R^s\to\R^m\times\R^s$ by
  $\mathcal{F}(\xi,\eta)=\big(-M^2 \xi,g(\xi,\eta)\big)$. Let
  $\Omega\subseteq [0,\infty)\times C_T^1(\R^m\times\R^s)$ be open and
  assume that $\deg(\mathcal{F},\Omega_\#)$ is well-defined and
  nonzero. Then there exists a connected set $\Gamma$ of nontrivial
  $T$-pairs for \eqref{DAE_ord2_vm} that meets
  $\mathcal{F}^{-1}(0)\cap\Omega$ and cannot be both bounded and
  contained in $\Omega$.
\end{theorem}

\begin{corollary}
  Let $f,g,A$ and $B$ be as above. Assume that $M:=A(t)\dot A(t)^\T$
  is constant and nonsingular.  Let $\Omega\subseteq [0,\infty)\times
  C_T^1(\R^m\times\R^s)$ be open. Assume that the set
  $[g(0,\cdot)]^{-1}(0)\cap\Omega_\#$ is nonempty and compact. Then
  there exists a connected set $\Gamma$ of nontrivial $T$-pairs for
  \eqref{DAE_ord2_vm} that meets $[g(0,\cdot)]^{-1}(0)\cap\Omega$ and
  cannot be both bounded and contained in $\Omega$.
\end{corollary}

\begin{theorem}\label{thm.DAE.2ord2}
  Let $f,g,A$ and $B$ be as above. Assume that $A(t)\dot A(t)^\T$ is
  identically zero and define
  $\omega\colon\R^m\times\R^s\to\R^m\times\R^s$ by
  \[
  \omega(\xi,\eta)=\left(\frac{1}{T}\int_0^TA(t)f\big(t,A(t)^\T\xi,B^{-1}(t)\eta,0,0\big)dt,
    g(\xi,\eta)\right).
  \]
  Let $\Omega\subseteq [0,\infty)\times C_T^1(\R^m\times\R^s)$ be open
  and assume that $\deg(\omega,\Omega_\#)$ is well-defined and
  nonzero. Then there exists a connected set $\Gamma$ of nontrivial
  $T$-pairs for \eqref{DAE_ord2_vm} that meets
  $\omega^{-1}(0)\cap\Omega$ and cannot be both bounded and contained
  in $\Omega$.
\end{theorem}

In the next example we consider the same time-dependent constraint as
in Example \ref{protante}, but in the case of second-order DAEs.
\begin{example}
  Let $f$ be as in Example \ref{protante}. Consider
  \[
  \left\{
    \begin{array}{l}
      \ddot x = \lambda f(t,x,y),\; \lambda\geq 0,\\
      y^3+y-x_1^2-x_2^2-(x_1\sin t+x_2\cos t)^2=0,
    \end{array}
  \right.
  \]
  where $x=(x_1,x_2)$.  Applying the coordinate transformation as
  described above we rewrite our DAE as follows:
  \[
  \left\{
    \begin{array}{l}
      \ddot \xi =\xi -2 M \dot \xi
      +\lambda A(t) f\left(t,A(t)^\T\xi,\eta\right),\; \lambda\geq 0,\\
      \eta^3+\eta-\xi_1^2-2\xi_2^2=0.
    \end{array}
  \right.
  \]
  Where
  \[
  M:= A\dot A^\T =
  \begin{pmatrix}
    0 & 1\\
    -1 & 0
  \end{pmatrix}
  \]
  so that $A\ddot A^\T=M^2=-I$.  Put
  $\F(p_1,p_2,q)=\big(p_1,p_2,q^3+q-p_1^2-2p_2^2\big)$, and let
  $\Omega=[0,\infty)\X C_T^1(\R^2\X\R)$. Then, since
  $\deg(\F,\R^3)=1\neq 0$, Theorem \ref{thm.DAE.2ord} yields an
  unbounded connected set of $2\pi$-periodic pairs emanating from
  $\big(0;(0,0;0)\big)\in [0,\infty)\X\R^2\X\R$ (regarded as a
  $2\pi$-pair).
\end{example}

\begin{remark} \label{rmk:second-const-per} As in the first order
  case, our coordinate transformation applies also to a slightly
  different situation. Consider the following DAE:
  \begin{equation}\label{IIord_conH}
    \left\{\begin{array}{l}
        \ddot x = H_1\dot x+ H_2  x+ \lambda f(t,x,y),\; \lambda\geq 0,\\
        g\big(A(t)x,B(t)y\big)=0,
      \end{array}\right.
  \end{equation} 
  where $A$, $B$, $f$ and $g$ are as in \eqref{DAE_ord2_vm} and $H_i$,
  $i=1,2$, are matrices that commute with $A$. Suppose, as above, that
  $M:=A(t)\dot A(t)^\T$ is constant (not necessarily invertible) and
  apply the transformation as indicated above. Equation
  \eqref{Iord_conH} becomes
  \begin{equation}\label{IIord_conH_trasf}
    \left\{
      \begin{array}{l}
        \ddot \xi =
        \big(H_1 M + H_2- M^2 \big)\xi + \big(H_1 - 2M)\dot\xi
        +\lambda F(t,\xi,\eta),\; \lambda\geq 0,\\
        g(\xi,\eta)=0.
      \end{array}
    \right.
  \end{equation}
  with $F$ as in \eqref{DAE_ord2_vf}, so that the results of
  Subsection \ref{GenIIord} are applicable to
  \eqref{IIord_conH_trasf}.
\end{remark}

\begin{remark}
  Let us consider the following second order DAE
  \begin{equation} \label{eq:additional-ord2-DAE} \left\{
      \begin{array}{l}
        \frac{d^2}{dt^2}(C(t) x) 
        = \lambda f(t,x,y,\dot x,\dot y),\; \lambda\geq 0,\\
        g\big(A(t)x,B(t)y\big)=0,
      \end{array}
    \right.
  \end{equation}
  where $f$, $A$ and $B$ are as in \eqref{DAE_ord2_vm} and $t\mapsto
  C(t)\in \Or(\R^m)$ is $C^2$ and $T$-periodic. We also assume that
  $C$ has the same property as $A$, that is, $C(t)\dot C(t)^\T$ is
  constant. Expanding the derivative on the left-hand side of the
  first equation in \eqref{eq:additional-ord2-DAE} and using the fact
  that $C(t)\in \Or(\R^m)$, for all $t\in\R$, we rewrite
  \eqref{eq:additional-ord2-DAE} as follows
  \begin{equation} \label{eq:addit-ord2-DAE-transf}
    \left\{ \begin{array}{l} \ddot x = - 2 C^\T \dot C\dot x -C^\T
        \ddot C x +
        \lambda C^\T f(t,x,y,\dot x,\dot y),\\
        g(Ax,By)=0
      \end{array} \right.
  \end{equation}  
  where, to keep the notation coincise, the explicit dependence on $t$
  of $A$, $B$ and $C$ is omitted. Proposition \ref{prop1} shows that
  $K_1:= C(t)^\T\dot C(t)$ is constant and, by Remark \ref{rem44}
  (\ref{r44.3}), it follows that $K_2:=C(t)^\T \ddot C(t)$ is constant
  as well being equal to $-K_1^2$. Hence,
  \eqref{eq:addit-ord2-DAE-transf} is of the form
  \eqref{IIord_conH}. Notice that if we assume that $A$ commutes with
  $K_1$, then it commutes with $K_2$ as well. In conclusion, if $K_1$
  commutes with $A(t)$ for all $t\in\R$, Remark
  \ref{rmk:second-const-per} applies with $H_1=-2K_1$ and $H_2=-K_2$.
\end{remark}

\section{Appendix: some lemmas of Matrix Analysis}
\label{sec:technical-lemmas}
\label{sec:tec-lemmas}
This section gathers, for reference purposes, a few simple facts
--possibly well-known-- concerning time dependent matrices.
\begin{lemma}\label{lemma1}
  Let $t\mapsto A(t)$ be a $C^2$ square-matrix-valued
  function. Suppose that the map $t\mapsto A(t)\dot A^\T(t)$ is
  constant. Then,
  \begin{align}
    & \ddot A(t)A(t)^\T = A(t)\ddot A(t)^\T \label{AdotAT}\\
    & \ddot A(t)A(t)^\T=-\dot A(t)\dot A(t)^\T \label{lemma1f}
  \end{align}
\end{lemma}
\begin{proof}
  For the sake of simplicity, we drop the explicit indication of the
  dependence of $A$ on $t$.

  Let us put $M=A\dot A^\T$. Then, $\dot AA^\T = (A\dot A^\T)^\T=M^\T$
  is also constant. Taking the derivative with respect to $t$ of both
  these relations, we get
  \begin{equation}\label{duerelazioni}
    \dot A\dot A^\T+A\ddot A^\T=0,\quad \ddot AA^\T+\dot A\dot A^\T=0.
  \end{equation}
  Hence,
  \[
  0=\dot A\dot A^\T+A\ddot A^\T-\ddot AA^\T-\dot A\dot A^\T=A\ddot
  A^\T-\ddot AA^\T,
  \]
  which implies \eqref{AdotAT}.

  {}From \eqref{duerelazioni} and \eqref{AdotAT} it follows
  \[
  0=\dot A\dot A^\T+A\ddot A^\T+\ddot AA^\T+\dot A\dot A^\T =2\dot
  A\dot A^\T+2\ddot AA^\T,
  \]
  whence the assertion.
\end{proof}

Observe that under the hypothesis of Lemma \ref{lemma1}, since
\[
\left[\ddot A(t)A(t)^\T\right]^\T =\ddot A(t)A(t)^\T,
\]
Equation \eqref{AdotAT} imply the simmetry of $A\ddot A^\T $.

\begin{lemma}\label{lemma2}
  Let $t\mapsto A(t)$ be a $C^1$ square-matrix-valued function. Assume
  that $A(t)$ is orthogonal for all $t$, then $\dot A(t)\dot A(t)^\T=
  -\big(A(t)\dot A(t)^\T\big)^2$.
\end{lemma}
\begin{proof}
  Differentiating the relation $A^\T A=I$ we obtain $\dot A^\T A= -
  A^\T\dot A$.  Multiplying this relation on the left by $\dot A$ and
  on the right by $A^\T$, we get
  \[
  \dot A\dot A^\T AA^\T = -\dot AA^\T\dot AA^\T.
  \]
  Since $AA^\T=I$ and $\dot A\dot A^\T=(\dot A\dot A^\T)^\T$,
  transposing yields
  \[
  \dot A\dot A^\T= -[(\dot AA^\T)^\T]^2 =-(A\dot A^\T)^2,
  \]
  as desired.
\end{proof}

Formula \ref{lemma1f} and Lemma \ref{lemma2} together yield the
following fact:

\begin{proposition}\label{prop1}
  Let $t\mapsto A(t)$ be a $C^2$ square-matrix-valued function. Assume
  that $A(t)$ is orthogonal for all $t$ and that the map $t\mapsto
  A(t)\dot A^\T(t)=:M$ is constant. Then $\ddot
  A(t)A(t)^\T=\big(A(t)\dot A(t)^\T\big)^2$ is constantly equal to
  $M^2$. In particular, if $A(t)\dot A^\T(t)$ is constant and
  nonsingular then so is $\ddot A(t)A(t)^\T$.
\end{proposition}

\begin{remark}\label{rem44}
  Replacing $A$ with $A^\T$, it is easy to verify that results
  analogous to Lemma \ref{lemma1}, Lemma \ref{lemma2} and Proposition
  \ref{prop1} hold if we assume the constancy of $A(t)^\T\dot A(t)$
  instead of that of $A(t)\dot A(t)^\T$.  Namely, if $t\mapsto A(t)$
  is a $C^2$ square-matrix-valued function such that $A(t)$ is
  orthogonal for all $t$ and the map $t\mapsto A(t)^\T\dot A(t)$ is a
  constant, then
  \begin{enumerate}
  \item $\ddot A(t)^\T A(t) = -\dot A(t)^\T \dot A(t)$ and $\ddot
    A(t)^\T A(t) = - A(t)^\T \ddot A(t)$;
  \item $\dot A(t)^\T \dot A(t)= - \big(A(t)^\T\dot A(t)\big)^2$;
  \item\label{r44.3} $\ddot A(t)^\T A(t)=\big(A(t)^\T\dot
    A(t)\big)^2=-A(t)^\T\ddot A(t)$.
  \end{enumerate}
  These facts should not surprise us in view of Proposition
  \ref{prop2} below.
\end{remark}

We conclude this technical section with a curious remark. As shown by
the following example:
\[
A(t)=\begin{pmatrix}
  0   &   0   &  \sin t & -\cos t\\
  0   &   0   &  \cos t &  \sin t\\
  \cos t & \sin t&    0    &    0   \\
  -\sin t & \cos t& 0 & 0
\end{pmatrix},
\]
even for matrix functions as in Proposition \ref{prop1}, one may have
\[
A(t)^\T\dot A(t)\neq A(t)\dot A(t)^\T.
\]
Nevertheless, one can prove the following fact:

\begin{proposition}\label{prop2}
  Let $t\mapsto A(t)$ be a $C^2$ square-matrix-valued function. Assume
  that $A(t)$ is orthogonal for all $t$. Then $A(t)^\T\dot A(t)$ is
  constant if and only if so is $A(t)\dot A(t)^\T$.
\end{proposition}

\begin{proof}
  Let us first prove that if $M:=A(t)\dot A(t)^\T$ is constant then
  $A(t)^\T\dot A(t)$ is constant as well. As above, for the sake of
  simplicity, we drop the explicit indication of the dependence of $A$
  on $t$.

  Clearly, we have $\dot A A^\T = M^\T$ and, since Proposition
  \ref{prop1} yields $A\ddot A^\T = M^2$, we also have $A\ddot
  A^\T=(M^2)^\T$. Now, using these facts we get
  \begin{multline*}
    \left[\frac{d}{dt}(A^T\dot A)\right]^\T = A^\T A\left[\frac{d}{dt}
      (A^\T\dot A)\right]^\T A^\T A
    % =\left[\dot A^\T\dot A+ A^\T\ddot A\right]^\T
    =A^\T A\left(\dot A^\T\dot A+\ddot A^\T A\right)A^\T A  \\
    =\left[\dot A^\T M\dot A+ A^\T (M^2)^\T A\right]A^\T A
    =\left[A^\T MM^\T + A^\T(M^2)^\T\right] A \\
    = A^\T\big(MM^\T+(M^2)^\T\big) A = A^\T\big(MM^\T+M^2\big)^\T A
  \end{multline*}
  Observe also that
  \[
  MM^\T+M^2= A\dot A^\T\left(\dot A A^\T + A\dot A^\T\right) =A\dot
  A^\T\left[\frac{d}{dt}(AA^\T)\right] = 0,
  \]
  because $AA^\T\equiv I$. Thus, $\frac{d}{dt}(A^T\dot A)=0$, which
  imply that $A(t)^\T\dot A(t)$ is a constant matrix.

  Conversely, if $A(t)^\T\dot A(t)$ is constant, a similar proof shows
  that $A(t)\dot A(t)^\T$ is constant too.
\end{proof}


\begin{thebibliography}{9}

% 
  % \bibitem{becafupe-} P.\ Benevieri, A.\ Calamai, M.\ Furi and M.\
  %   P.\ Pera, \emph{Delay differential equations on manifolds and
  %   applications to motion problems for forced constrained
  %   systems}, Z. Anal. Anwend. 28 (2009), no. 4, 451-474.
  % %
\bibitem{BiSpa2011} L.\ Bisconti and M. Spadini, \emph{On a class of
    differential-algebraic equations with infinite delay}, Electronic
  Journal of Qualitative Theory of Differential Equations, 2011,
  No. 81, 1-21.
%  
\bibitem{BiSpa2012} L.\ Bisconti and M. Spadini, \emph{Corrigendum to
    On a class of differential-algebraic equations with infinite
    delay}, Electronic Journal of Qualitative Theory of Differential
  Equations, 2012, No. 97, 1-5.
  % %
  % \bibitem{lavesp} P.\ Benevieri, A.\ Calamai, M.\ Furi and M.\ P.\
  %   Pera, \emph{On General Properties of Retarded Functional
  %   Differential Equations on Manifolds}, in preparation
  % % %
  % \bibitem{brauer-cast} F.\ Brauer and
  %   C. Castillo-Ch$\acute{\textrm{a}}$vez, \emph{Mathematical Models
  %   in Population Biology and Epidemiology}, Springer-Verlag,
  %   Texts in Applied Mathematics vol. 40, 2000.
  % %
  % \bibitem{burt-1} T.\ A.\ Burton, \emph{Stability by Fixed Point
  %   for Functional Differential Equations}, Dover Publications,
  %   2006.
  % %
\bibitem{cala} A.\ Calamai, \emph{Branches of harmonic solutions for a
    class of periodic differential-algebraic equations}, Comm.\ Appl.\
  Analysis 15 (2011), no. 2-4, 273-282.
%
\bibitem{CS1}A.\ Calamai and M.\ Spadini, \textit{Branches of forced
    oscillations for a class of constrained ODEs: a topological
    approach}.  NoDEA Nonlinear Differential Equations Appl. 19
  (2012), no. 4, 383-399.
%  
\bibitem{CS2}A.\ Calamai and M.\ Spadini, \textit{Periodic
    perturbations of constrained motion problems on a class of
    implicitly defined manifolds}, Preprint.
%
\bibitem{fupespa} M.\ Furi, M.\ P.\ Pera and M.\ Spadini, \emph{The
    fixed point index of the Poincar\'e operator on differentiable
    manifolds}, Handbook of topological fixed point theory, Brown
  R. F., Furi M., G\'orniewicz L., Jiang B. (Eds.), Spinger, 2005.
  % %
  % \bibitem{fupespa2} M.\ Furi, M.\ P.\ Pera and M.\ Spadini,
  %   \emph{Periodic solutions of functional differential
  %   perturbations of autonomous differential equations on
  %   manifolds}, to appear in Comm.\ Appl.\ Analysis.
  % %
  % \bibitem{fuspa3} M.\ Furi and M.\ Spadini, \emph{On the set of
  %   harmonic solutions of periodically perturbed autonomous
  %   differential equations on manifolds}, Nonlinear Anal.\ Vol. 29
  %   (1997), no. 8, 963-970.
  % % %
\bibitem{mgerdin} M.\ Gerdin \emph{Identification and Estimation for
    Models Described by Differential-Algebraic Equations}, Department
  of Electrical Engineering Link\"opings universitet, SE-581 83
  Link\"oping, Sweden, Link\"oping 2006.
  % %
\bibitem{GoVL} G.\ H.\ Golub and C.\ F.\ Van Loan, \emph{Matrix
    computations}, 3$^\mathrm{rd}$ edition, J.\ Hopkins Univ.\ Press,
  Baltimore 1996.
  % %
  % \bibitem{difftop} V.\ Guillemin and A.\ Pollack,
  %   \emph{Differential-Topology}, Prentice-Hall Inc., Englewood
  %   Cliffs, New Jersey, 1974.
  % % %
  % \bibitem{dae-1994} P.\ Kunkel and V.\ Mehrmann, \emph{Canonical
  %   forms for linear differential-algebraic equations with
  %   variable coefficients}, Journal of Computational and Applied
  %   Mathematics, 56(3): 225-251, December 1994.
  % %
  % \bibitem{dae-2001} P.\ Kunkel and V.\ Mehrmann, \emph{Analysis of
  %   over and underdetermined nonlinear differential-algebraic
  %   systems with application to nonlinear control problems},
  %   Mathematics of Control, Signals, and Systems, 14(3):233-256,
  %   2001.
  % %
\bibitem{dae} P.\ Kunkel and V.\ Mehrmann,
  \emph{Differential-Algebraic Equations: Analysis and Numerical
    Solution}, EMS Textbooks in Mathematics, 2006.
  % %
  % \bibitem{milnor} J.\ W.\ Milnor, \emph{Topology from the
  %   differentiable viewpoint}, Univ.\ press of Virginia,
  %   Charlottesville, 1965.
  % % %
  % \bibitem{RR} P.\ J.\ Rabier and W.\ C.\ Rheinbolt,
  %   \emph{Nonholonomic Motion of Rigid Mechanical Systems from a DAE
  %   Viewpoint.} SIAM, Philadelphia, 2000.
  % %
\bibitem{RR0} P.\ J.\ Rabier and W.\ C.\ Rheinbolt, \emph{Theoretical
    and numerical analysis of differential-algebraic equations},
  Handbook of Numerical Analysis Volume 8, 2002, Pages 183-540,
  Solution of Equations in $\R^n$ (Part 4), Techniques of Scientific
  Computing (Part4), Numerical Methods for Fluids (Part 2). Editors
  P.\ G.\ Ciarlet and J.\ L.\ Lions.  Elsevier Science, 2002.
%
\bibitem{R:1988} W.\ C.\ Rheinboldt. \emph{Differential-algebraic
    systems as differential equations on manifolds.}  Math. Comp., 43
  : 473--482, 1984.
%
  % \bibitem{Sch} S.\ Schulz, \emph{Four Lectures on
  %   Differential-Algebraic Equations}, Univ.\ of Auckland,
  %   Department of Mathematics - Research Reports-497 (2003).
%
\bibitem{spaDAE} M.\ Spadini, \emph{A note on topological methods for
    a class of Differential-Algebraic Equations}, Nonlinear Anal. 73
  (2010), no. 4, 1065-1076.
\end{thebibliography}
\end{document}